\documentclass[11pt,twoside,reqno,centertags]{amsart}
\usepackage{amsfonts}
\usepackage{color,enumitem,graphicx}
\usepackage{subfigure}
\usepackage[colorlinks=true,urlcolor=blue,
citecolor=red,linkcolor=blue,linktocpage,pdfpagelabels,
bookmarksnumbered,bookmarksopen]{hyperref}

  \usepackage{amsmath,amsthm,amsfonts,amssymb}

\begin{document}
\title{ Isolated singularities of positive solutions for Choquard equations in sublinear  case    }
\date{}
\maketitle

\vspace{ -1\baselineskip}

{\small
\begin{center}

\medskip

  {\sc  Huyuan Chen}
  \medskip

 Department of Mathematics, Jiangxi Normal University,\\
Nanchang, Jiangxi 330022, PR China\\
 Email: chenhuyuan@yeah.net\\[10pt]
 {\sc  Feng Zhou}
   \medskip

Center for PDEs and Department of Mathematics, East China Normal University,\\
 Shanghai, 200241, PR China\\
 Email: fzhou@math.ecnu.edu.cn\\[10pt]

\end{center}
}

\medskip

\begin{quote}
{\bf Abstract.}
Our purpose of this paper is to study the isolated singularities of  positive solutions  to  Choquard equation in the sublinear case $q \in (0,1)$
 $$\displaystyle \ \  -\Delta   u+  u =I_\alpha[u^p] u^q\;\;
  {\rm in}\;  \mathbb{R}^N\setminus\{0\},
 \;\;  \displaystyle   \lim_{|x|\to+\infty}u(x)=0,
$$
where $p >0,  N \geq 3, \alpha \in (0,N)$ and
$I_{\alpha}[u^p](x) = \int_{\mathbb{R}^N} \frac{u^p(y)}{|x-y|^{N-\alpha}}dy$ is the Riesz potential, which appears as a nonlocal term in the equation.
We investigate the nonexistence and existence of isolated singular solutions of
Choquard equation under different range of the pair of exponent $(p,q)$.
Furthermore, we obtain qualitative properties for the minimal singular solutions of the equation.

\end{quote}

 \renewcommand{\thefootnote}{}

\newcommand{\N}{\mathbb{N}}
\newcommand{\R}{\mathbb{R}}
\newcommand{\Z}{\mathbb{Z}}

\newcommand{\cA}{{\mathcal A}}
\newcommand{\cB}{{\mathcal B}}
\newcommand{\cC}{{\mathcal C}}
\newcommand{\cD}{{\mathcal D}}
\newcommand{\cE}{{\mathcal E}}
\newcommand{\cF}{{\mathcal F}}
\newcommand{\cG}{{\mathcal G}}
\newcommand{\cH}{{\mathcal H}}
\newcommand{\cI}{{\mathcal I}}
\newcommand{\cJ}{{\mathcal J}}
\newcommand{\cK}{{\mathcal K}}
\newcommand{\cL}{{\mathcal L}}
\newcommand{\cM}{{\mathcal M}}
\newcommand{\cN}{{\mathcal N}}
\newcommand{\cO}{{\mathcal O}}
\newcommand{\cP}{{\mathcal P}}
\newcommand{\cQ}{{\mathcal Q}}
\newcommand{\cR}{{\mathcal R}}
\newcommand{\cS}{{\mathcal S}}
\newcommand{\cT}{{\mathcal T}}
\newcommand{\cU}{{\mathcal U}}
\newcommand{\cV}{{\mathcal V}}
\newcommand{\cW}{{\mathcal W}}
\newcommand{\cX}{{\mathcal X}}
\newcommand{\cY}{{\mathcal Y}}
\newcommand{\cZ}{{\mathcal Z}}

\newcommand{\abs}[1]{\lvert#1\rvert}
\newcommand{\xabs}[1]{\left\lvert#1\right\rvert}
\newcommand{\norm}[1]{\lVert#1\rVert}

\newcommand{\loc}{\mathrm{loc}}
\newcommand{\p}{\partial}
\newcommand{\h}{\hskip 5mm}
\newcommand{\ti}{\widetilde}
\newcommand{\D}{\Delta}
\newcommand{\e}{\epsilon}
\newcommand{\bs}{\backslash}
\newcommand{\ep}{\emptyset}
\newcommand{\su}{\subset}
\newcommand{\ds}{\displaystyle}
\newcommand{\ld}{\lambda}
\newcommand{\vp}{\varphi}
\newcommand{\wpp}{W_0^{1,\ p}(\Omega)}
\newcommand{\ino}{\int_\Omega}
\newcommand{\bo}{\overline{\Omega}}
\newcommand{\ccc}{\cC_0^1(\bo)}
\newcommand{\iii}{\opint_{D_1}D_i}

\numberwithin{equation}{section}

\vskip 0.2cm \arraycolsep1.5pt
\newtheorem{lemma}{Lemma}[section]
\newtheorem{theorem}{Theorem}[section]
\newtheorem{definition}{Definition}[section]
\newtheorem{proposition}{Proposition}[section]
\newtheorem{remark}{Remark}[section]
\newtheorem{corollary}{Corollary}[section]

\setcounter{equation}{0}
\section{Introduction}

This is a continuation of the work \cite{CZ} on the study of the isolated singularities of positive solutions to Choquard equation
\begin{equation}\label{eq 1.1}
  \arraycolsep=1pt
\begin{array}{lll}
 \displaystyle\ \quad  -\Delta   u+  u =I_\alpha[u^p] u^q\quad
  {\rm in}\quad  \mathbb{R}^N\setminus\{0\},\\[2mm]
 \phantom{    }
 \;\; \displaystyle   \lim_{|x|\to+\infty}u(x)=0,
\end{array}
\end{equation}
where $u$ is a classical solution in $\mathbb{R}^N\setminus\{0\}$,  $p>0,\, q > 0$, $N\ge3$, $\alpha\in(0,N)$ and
$$I_\alpha[u^p](x)=\int_{\mathbb{R}^N}\frac{u^p
(y)}{|x-y|^{N-\alpha}}\, dy.$$
  We note that it is natural to assume that $\alpha>0$, otherwise $I_\alpha[u^p]$ will be infinite in whole $\R^N$ and here  $I_\alpha(x)=|x|^{\alpha-N}$ is the Riesz potential with the order $\alpha-N<0$, which is related to the fractional Laplacian  when $\alpha\in(0,2)$ and it is a nonlocal operator.
 Here $u$ is said to be a classical solution of (\ref{eq 1.1}) if $u\in C^2(\R^N\setminus\{0\})$, $I_\alpha[u^p]$ is well-defined in $\R^N\setminus\{0\}$ and $u$ satisfies  (\ref{eq 1.1}) pointwisely. We have developed in \cite{CZ} the method of Lions \cite{L} (e.g. \cite{BL}) to classify the isolated singularities of the Choquard  equation (\ref{eq 1.1}) for the case of $q \geq 1$.
The essential point to set $q\ge1$ in \cite{CZ} is due to the fact that the positive solutions of (\ref{eq 1.1}) decays exponentially and
this fact guarantees the existence of isolated singular solutions. When $0< q<1$, the situation of the nonlinearity becomes
subtle, because the decay at infinity  is no longer exponential and this causes many difficulties for the classification of the isolated singularities and for the construction of singular solutions of the equation.

The Choquard equation is considered as an approximation to Hartree-Fock theory for a one component plasma, which is a semilinear Schr\"{o}dinger-Newton type equation proposed by P. Choquard (for $N=3$, $\alpha=p=2$ and $q=1$) and  explained in Lieb and Lieb-Simon's papers \cite{L0,LS} respectively.
The Choquard type equations also arise in the  physics of multiple-particle systems (\cite{G}).  Furthermore, the Choquard type equations appear to a prototype of the nonlocal problems, which play a fundamental role in some Quantum-mechanical and non-linear optics (see e.g. \cite{GL,O}). As far as we know, the most mathematical results about the Choquard equations  are known for the case of $q =p-1$.
More precisely, we consider the problem in the variational setting:
$$\arraycolsep=1pt
\begin{array}{lll}
 \displaystyle   -\Delta   u+  u =I_\alpha[|u|^p] |u|^{p-2}u\quad
  {\rm in}\quad  \mathbb{R}^N,\\[2mm]
 \phantom{  --- - }
   u\in H^{1} (\R^N),
\end{array}
$$
that is, for the solutions which correspond to the critical points of a functional defined on the Sobolev space
$H^1(\mathbb{R}^N)$ and $p \in  (1, +\infty)$ satisfies
\begin{equation}\label{gs}
 \frac{N-2}{N+\alpha} \leq \frac{1}{p} \leq \frac{N}{N+\alpha}.
 \end{equation}
 Then the equation has a groundstate solution $u \in H^1(\mathbb{R}^N)$ if $p$ satisfies strictly the inequality in (\ref{gs}) (see \cite{MVS}, \cite{MZ}, \cite{L0} and in particular \cite{MVS1} which is a survey for the Choquard type equations, the related problems and the related references therein).

Our aim  in this article focuses on the study of the Choqurad equation
 (\ref{eq 1.1}) when $q\in(0,1)$.
Without special explanation, we always assume in the sequel that
$${\bf N\ge3,\ \,  \alpha\in(0,N),\ \, p>0\, \ {\rm and}\ \, q\in(0,1).}$$
We emphasize that the problem we treated is in general non-variational and the solutions are considered in the distributional sense.
Our first result states the nonexistence of positive solution of (\ref{eq 1.1}).
 \begin{theorem}\label{teo 0}
Assume that
    \begin{equation}\label{q}
 (1-\frac\alpha N)p+q<1\quad{\rm and}\quad  p+q<1+\frac{\alpha}{N-2}.
 \end{equation}
Then  problem (\ref{eq 1.1}) has no any nonnegative nontrivial solution.
\end{theorem}

Notice that in
\cite{MVS}, the authors showed an elegant  results on the groundstates for the case $q=p-1$ and gave the optimal range of $p$ for the existence of the groundstates solution. In fact, they proved that (\ref{eq 1.1})
has no nontrivial solution when $p\ge \frac{N+\alpha}{N-2}$ or $p\le 1+\frac{\alpha}{N}$, at least for groundstates solutions by applying Poho\v{z}aev identity. In our case, the distributional solution are weaker than the variational solution, and we consider the pair exponent $(p,q)$
in the planar domain, so the range of $(p,q)$ for the existence or nonexistence of distributional solution are delicate. It would be larger than the one in variational sense.

Next we present the classification of singularities of (\ref{eq 1.1}), inspired by \cite{CZ}.

\begin{theorem}\label{teo 1}
Let $u$ be a nonnegative classical solution of (\ref{eq 1.1}) and
 \begin{equation}\label{1.2}
(1-\frac\alpha N)p+q\ge1\qquad{\rm or}\qquad p+q \ge 1+\frac{\alpha}{N-2}.
 \end{equation}
We assume that $(i)$ $u^p\in L^1(\R^N)$  or  $(ii)$ $u^p\in L^1_{loc}(\R^N)$ and there exist $\bar \alpha\in(\alpha,\,N)$ and $c_0>0$ such that
$$u^p(x)\le c_0|x|^{-\bar \alpha},\quad\forall\, |x|>1.$$
 Then  there exists $k\ge0$
  such that  $u$ is a solution of
 \begin{equation}\label{eq 1.2}
    \arraycolsep=1pt
\begin{array}{lll}
 \displaystyle \ \  -\Delta u+u=I_\alpha[u^p]u^q+k\delta_0\quad
 &{\rm in}\quad \R^N,\\[2mm]
 \phantom{   }
 \;\; \displaystyle   \lim_{|x|\to\infty}u(x)=0,
\end{array}
\end{equation}
in the sense of distribution,
that is,
 \begin{equation}
     \int_{\R^N} \left(u (-\Delta  \xi)+u\xi-I_\alpha[u^p]u^q\xi\right)\, dx=k\xi(0),\quad \forall \xi\in C^\infty_c(\R^N),
\end{equation}
where $C^\infty_c(\R^N)$ is the space of all the functions in  $C^\infty(\R^N)$   with  compact support.

Furthermore,
$(i)$  when
\begin{equation}\label{1.3}
 p+q\ge \frac{N+\alpha}{N-2}\quad{\rm or}\quad p\ge \frac{N}{N-2},
\end{equation}
then $k=0$;

$(ii)$ when
\begin{equation}\label{1.4}
 p+q< \frac{N+\alpha}{N-2}\quad{\rm and}\quad p< \frac{N}{N-2}.
\end{equation}
  If $k=0$,  then $u$ is a classical solution of
  \begin{equation}\label{eq 1.3}
  \arraycolsep=1pt
\begin{array}{lll}
 \displaystyle \quad\  -\Delta  u+u=I_\alpha[u^p]u^q\quad
 &{\rm in}\quad \R^N,\\[2mm]
 \phantom{   }
 \;\;\displaystyle   \lim_{|x|\to+\infty}u(x)=0;
\end{array}
\end{equation}
 if $k>0$, then $u$ satisfies that
\begin{equation}\label{1.5}
 \lim_{|x|\to0^+} u(x)|x|^{N-2}=c_{N} k,
\end{equation}
where $c_N$ is the normalized constant depending only on $N$.
\end{theorem}

 We remark that Theorem \ref{teo 1} part $(i)$ shows that in the case (\ref{1.3}),
the singularities of positive solutions of (\ref{eq 1.1}) are not visible in the distribution sense by the Dirac mass and it is open but interesting to consider the singularities in this case. The Choquard equation (\ref{eq 1.1}) could be divided into a system of equations with the Laplacian in the
linear part of the first equation and fractional Laplacian in the second one. We remark that the Dirac mass only appears
in the first equation for the system. This is different for considering the system directly. As in \cite{CZ}, the basic tool we used
to connect the singular solutions of
elliptic equation in punctured domain and the solutions of corresponding elliptic equation in the distributional
 sense is an early result \cite{BL} due to
Brezis and Lions on the study of isolated singularities.
It is well known that the study of singularities of semilinear elliptic equations is a major subject in PDEs and they have been investigated widely for decades (see for example \cite{A,BB,BV,BL,CGS,
DDG,DGW,GS,MV,V}).

Note that in \cite{GT}, the authors gave some very interesting results on the behavior of the
isolated singularity for positive subsolution $u\in C^2(\R^N\setminus\{0\})$ of the Choquard type inequality
$$ 0 \leq -\Delta  u =I_\alpha[u^p]u^q \quad
  {\rm in}\quad  B_1(0) \setminus\{0\},
$$
where the methods on analysis of the singularities are also different. In particular the solution for their case is a superhamonic function and the operator that we consider here is $-\Delta + Id$.

As we have mentioned above, the case of $0< q<1$ is delicate, at least for the methods we have used in \cite{CZ} for the classification of isolated singularities and existence of solutions. In fact, as showed by Moroz and Van Schaftingen in \cite{MVS},\cite{MVS0} (see e.g. \cite{MVS1}), the solutions of the Choquard equation may have polynomial decay at infinity,  which makes the classification of singularities difficult. The polynomial decay at infinity can not guarantee that $I_\alpha[u^p] $ is well defined and then it may cause the nonexistence of solutions for the equation. However, we can still prove the existence results stated as follows.

\begin{theorem}\label{teo 2}
Assume that $(p,q)\in(0,\frac{N}{N-2})\times (0,1)$ satisfies (\ref{1.2}) and $ p+q <\frac{N+\alpha}{N-2}$.

 Then there exists $k^*>0$   such that
 for $ k\in(0, k^*)$, problem (\ref{eq 1.2}) admits a minimal positive solution $u_{k}$ in the distribution sense, which is a classical solution of (\ref{eq 1.1}) and satisfies (\ref{1.5}).

Moreover, if
\begin{equation}\label{1.6}
 (1-\frac\alpha N)p+q>1\quad{\rm and}\quad  p+q<  \frac{N+\alpha}{N-2},
\end{equation}
the solution $u_k$ has  the decay at infinity as
\begin{equation}\label{1.8}
 \lim_{|x|\to+\infty}u_k(x)|x|^{\frac{N-\alpha}{1-q}}=\norm{u_k}^{\frac{p}{1-q}}_{L^1(\R^N)};
\end{equation}
if
\begin{equation}\label{1.9}
 (1-\frac\alpha N)p+q\le1\quad{\rm and}\quad  1+\frac{\alpha}{N-2}\le p+q<  \frac{N+\alpha}{N-2},
\end{equation}
the solution $u_k$ has  the decay at infinity as
\begin{equation}\label{1.10}
 \limsup_{|x|\to+\infty}u_k(x)|x|^{\max\{N-2,\frac{N-\alpha}{1-q}\}}\le k\quad{\rm and}\quad  \liminf_{|x|\to+\infty}u_k(x)|x|^{\max\{N-2,\frac{N-\alpha}{1-q}\}}>0.
\end{equation}
\end{theorem}

 \begin{figure}
   \centering
  \subfigure{
 \begin{minipage}{65mm}
  \includegraphics[scale=0.033]{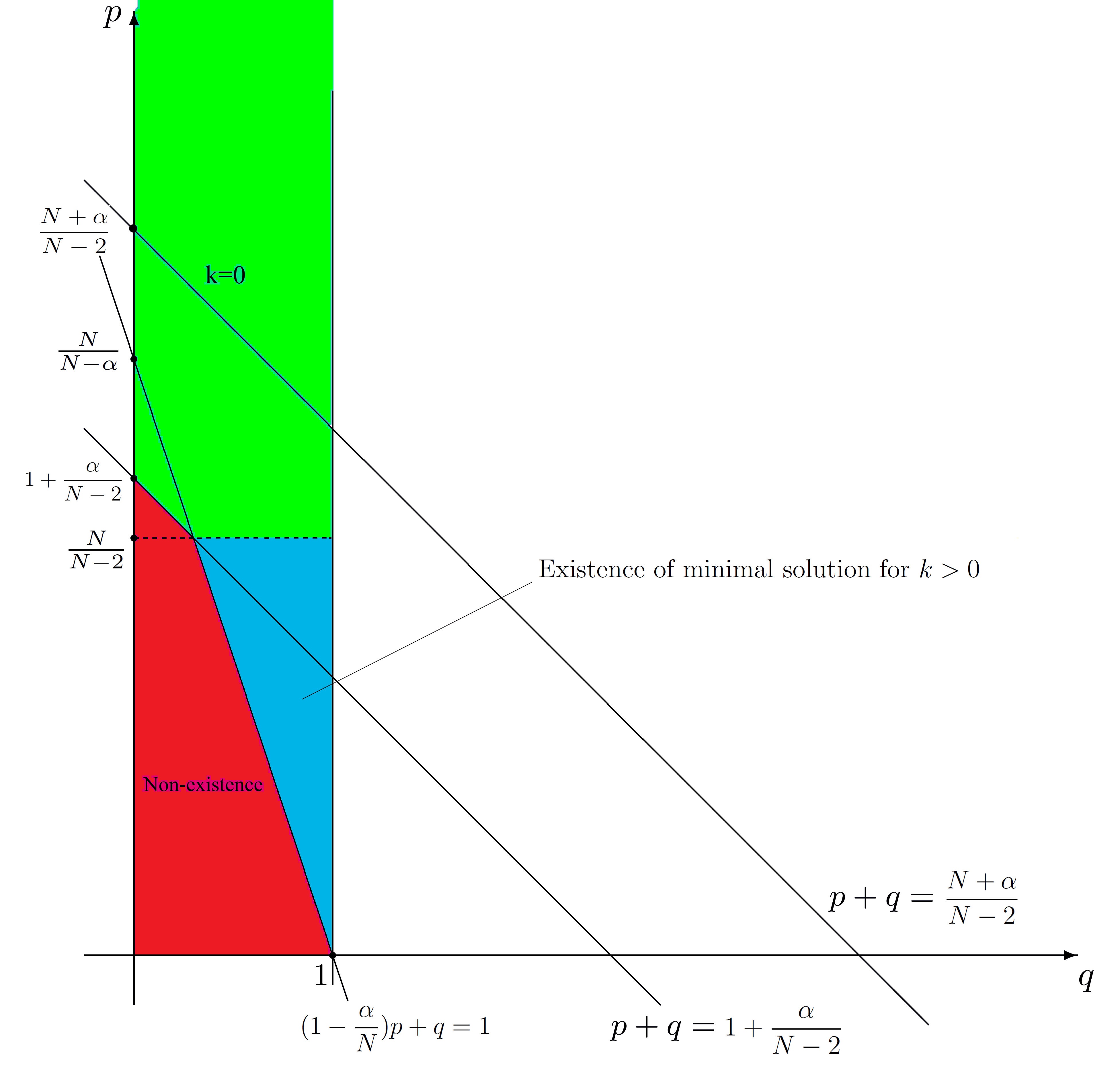}
   
   \qquad{\small  Figure 1. $\alpha\in(\sqrt{2N},N)$}
  \end{minipage}
 }
   \subfigure{
   \begin{minipage}{60mm}
  \includegraphics[scale=0.033]{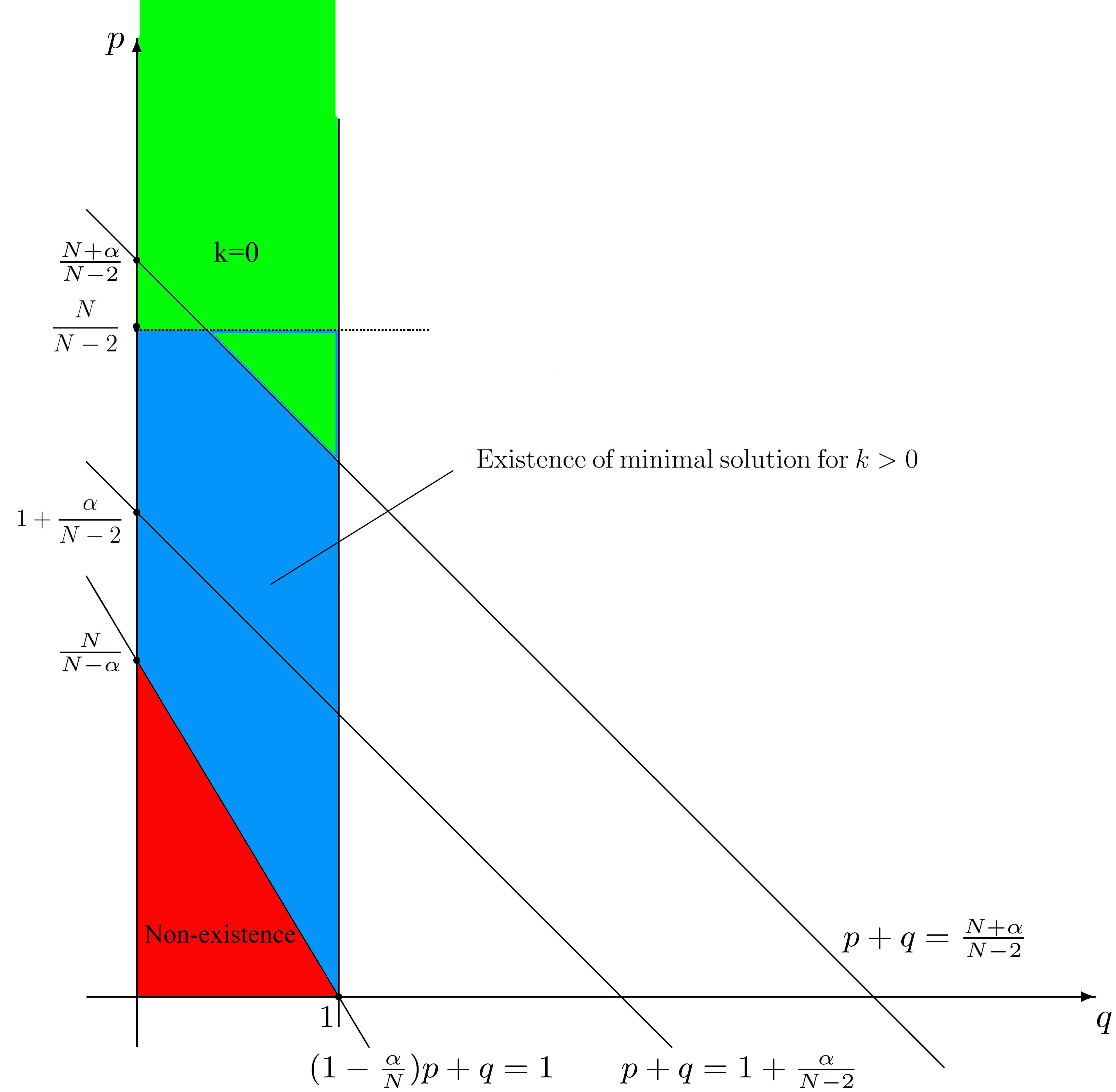}
  
   \qquad {\small Figure 2. $\alpha \in (0,2)$}
  \end{minipage}
  }

 {\footnotesize     The red regions of $(p,q)$ represent the nonexistence showed in Theorem \ref{teo 0}, the blue and the green ones are related to (\ref{1.4}) and (\ref{1.3}) respectively in Theorem \ref{teo 1}  and Theorem \ref{teo 2} proves the existence of (\ref{eq 1.2})\\ \ \ with $k>0$ in the blue regions.\hfill  }
\end{figure}

For the proof of this theorem, the standard iterating procedure (see e.g. \cite{V}) is adopted to obtain the existence of singular solutions of (\ref{eq 1.1}) and the main part is to construct a suitable upper bound for the procedure. Furthermore, this upper bound provides the first estimate in (\ref{1.10}), and combining (\ref{1.8}), we may conclude that
 $\norm{u_k}_{L^1(\R^N)}\le k^{\frac{1-q}{p}}.$ The proof of  (\ref{1.8}) is motivated by \cite{MVS},  where the authors provided the decay estimate
  for  the positive ground state solution of
$$
  \arraycolsep=1pt
 \begin{array}{lll}
  \displaystyle\ \quad  -\Delta   u+  u =I_\alpha[u^p] u^{p-1}\quad
   {\rm in}\quad  \mathbb{R}^N,\\[2mm]
  \phantom{    }
  \;\; \displaystyle   \lim_{|x|\to+\infty}u(x)=0.
 \end{array}
$$
Our method is to set for $u^p\le |x|^{-\beta}$ with $\beta>N$. However, $u^p$ is even no longer in $L^1(\R^N)$  if (\ref{1.6}) fails. In fact,
 from Lemma \ref{pr e 1} below, we have that
 $u_k\ge c_1|x|^{\tau_0}$ for $|x|> 1$,
 where $c_1>0$ and $\tau_0=-\max\{N-2,\frac{N-\alpha}{1-q}\}$. The upper bound for the decay at infinity could be seen in Lemma \ref{lm 4.3}
 where we construct super solutions to  control iterating procedure for the existence of $u_k$.
When $\alpha\in(0,2]$, conditions (\ref{q}) and (\ref{1.2}) could be reduced into  $(1-\frac\alpha N)p+q<1$ and
$(1-\frac\alpha N)p+q\ge1$ respectively.

The rest of this paper is organized as follows. Section 2  is devoted
to the non-existence of positive solution of (\ref{eq 1.1}). In  Section 3 we classify the singularities for equation (\ref{eq 1.1}) when the pair of exponent $(p,q)$ satisfies (\ref{1.2}), and prove the existence of singular solutions and show that the decay (\ref{1.8}) holds.

\hskip0.5cm

 \section{Nonexistence}

We prove the nonexistence of weak solution of (\ref{eq 1.1}) by contradiction. Assume that problem (\ref{eq 1.1})
admits a  nonnegative nontrivial solution $u$, and then  we will obtain a contradiction by the the blowing up phenomena derived from its decay at infinity.
 Let \begin{equation}\label{t0}
 \tau_0=-\max\{\frac{N-\alpha}{1-q},N-2\}
\end{equation}
and denote by $\{\tau_j\}_j$  the sequence
\begin{equation}\label{2.1}
 \tau_j=\frac{\alpha}{1-q}+\frac{p}{1-q}\tau_{j-1}\quad{\rm for}\quad  j=1,2,3\cdots.
\end{equation}
We start with the following
\begin{lemma}\label{lm 2.1}
Assume that $(p,\, q)$ satisfies (\ref{q}),
  then $\{\tau_j\}_j$ is an increasing sequence and
there exists $j_0\in\N $  such that
\begin{equation}\label{2.3}
\tau_{j_0-1}<-\frac{\alpha}{p} \quad {\rm and}\quad \tau_{j_0}\ge-\frac{\alpha}{p}.
\end{equation}
\end{lemma}
\begin{proof}
Since $q \in (0,1)$, we have that
$\tau_1-\tau_0>0$ is equivalent to
\begin{equation}\label{00000}
 \tau_0(1-p-q)<\alpha.
\end{equation}
We observe that (\ref{00000}) is obvious for $p+q\le 1$,
 and when $p+q>1$, (\ref{00000}) holds if
$$\frac{N-\alpha}{1-q}<\frac{\alpha}{p+q-1}\quad{\rm and}\quad N-2<\frac{\alpha}{p+q-1},$$
which is exactly equivalent to condition (\ref{q}).
Notice also that
\begin{equation}\label{2.2}
\tau_j-\tau_{j-1} = \frac{p}{1-q}(\tau_{j-1}-\tau_{j-2})=\left(\frac{p}{1-q}\right)^{j-1} (\tau_1-\tau_0),
\end{equation}
which implies that the sequence $\{\tau_j\}_j$ is  increasing under the condition (\ref{q}).
If $\frac{p}{1-q}\ge1 $, the conclusion is clear. If $\frac{p}{1-q}\in(0,1)$,
it deduces from (\ref{2.2}) that
\begin{eqnarray*}
\tau_j  &=&  \frac{1-\left(\frac{p}{1-q}\right)^j}{1-\frac{p}{1-q}}(\tau_1-\tau_0)+\tau_0\\
&\to&\frac{1-q}{1-p-q}(\tau_1-\tau_0)+\tau_0=\frac{\alpha}{1-p-q}>0,\quad {\rm as}\quad j\to+\infty,
\end{eqnarray*}
 then there exists $j_0>0$ satisfying (\ref{2.3}).
\end{proof}

 \begin{proposition}\label{pr e 1}
Let $u$ be a nonnegative and nontrivial classical solution of (\ref{eq 1.1}),
then there exists $b_0>0$ such that
\begin{equation}\label{e 1}
u(x)\ge b_0 |x|^{\tau_0} \quad {\rm for} \quad |x|\ge 1,
\end{equation}
 where $\tau_0$ is given by (\ref{t0}).
 \end{proposition}
\begin{proof}
Since $u$ is nonnegative, then $I_\alpha[u^p]u^q\ge 0$ and by the Strong Maximum Principle we have that
$ u>0$ in$ \R^N\setminus\{0\} $ since $u$ is a nontrivial solution.

{\it Step 1. } We claim that there exists $c_2>0$ such that
\begin{equation}\label{st1}
 I_\alpha[u^p](x)\ge c_2\min\{1,\, |x|^{\alpha-N}\},\qquad\forall x\in \R^N\setminus\{0\}.
\end{equation}
In fact, fix a point $x_0$ with $|x_0|=1$,  there exists $c_3>0$ such that
$$u^p\ge c_3\quad{\rm in}\quad \overline{B_{\frac12}(x_0)},$$
then for $|x|\ge2$ and $|y|<\frac32$, we have that $|x-y|\le |x|+|y|\le 2|x|$ and then
\begin{eqnarray*}
  I_\alpha[u^p](x) &\ge&  c_3\int_{B_{\frac12}(x_0)} \frac{1}{|x-y|^{N-\alpha}}dy \\
    &\ge& c_3 (2|x|)^{\alpha-N}  \int_{B_{\frac12}(x_0)}dy
  \ge  c_4 |x|^{\alpha-N}.
\end{eqnarray*}
For $|x|\le2$ and $|y|<\frac32$, we have that $|x-y| \le 4$ and
  \begin{equation}\label{2.4-3}
  I_\alpha[u^p](x) \ge c_3 \int_{B_{\frac12}(x_0)} \frac{1}{|x-y|^{N-\alpha}} dy\ge 4^{\alpha-N}c_3|B_{1/2}(0)|.
  \end{equation}
Therefore,  (\ref{st1}) holds.

 Let
 $$v_1(x)=\arraycolsep=1pt\left\{
\begin{array}{lll}
 \displaystyle   c_4  |x|^{\tau_0} \quad
 &{\rm if}\quad  |x|\ge 1,\\[2mm]
 \phantom{  }
  0  \quad
 &{\rm if}\quad  |x|< 1,
\end{array}
\right.
$$
where
$$c_4=\min\left\{(c_2 q)^{\frac1{1-q}},\, \min_{x\in \partial B_1(0)} u(x)\right\}.$$

{\it Step 2. } We show that $v_1$ satisfies
 \begin{equation}\label{eq 2.01}
 \displaystyle    -\Delta v_1\le c_2|x|^{\alpha-N}v_1^q-v_1\quad
  {\rm in}\quad  \mathbb{R}^N\setminus \overline{B_1(0)}.
 \end{equation}
Since $\tau_0\le 2-N$, then for $|x|>1$,
$$ -\Delta v_1\le 0$$
and
\begin{eqnarray*}
  c_2|x|^{\alpha-N}v_1^q-v_1 & =& v_1^q[c_2|x|^{\alpha-N}-v_1^{1-q}(x)]  \\
  &\ge&  v_1^q \left[c_2|x|^{\alpha-N}-c_2q |x|^{\tau_0(1-q)}\right]
   \ge 0.
\end{eqnarray*}

{\it Step 3. } We prove finally that
\begin{equation}\label{2.4-2}
 u\ge v_1\quad {\rm in}\ \ \R^N\setminus \overline{B_1(0)}.
\end{equation}

Let $f(r,t)=c_2r^{\alpha-N} t^q-t$, then
$t\mapsto f(r,t)$ is decreasing in
$(0,\, (c_2 q)^{\frac1{1-q}}  r^{\tau_0})$.
In fact,
$$\partial _t f(r,t)=c_2q r^{\alpha-N} t^{q-1}-1<0\quad {\rm for }\quad t \in (0,\, (c_2 q)^{\frac1{1-q}}  r^{\tau_0}).$$

 If (\ref{2.4-2}) fails, since $u\ge v_1$ on $\partial B_1(0)$,  we may assume that there exists some point $x$ in $\R^N\setminus \overline{B_1(0)}$ such that
$u(x)-v_1(x)<0$, denoting
$$l_0=\inf_{x\in\R^N\setminus B_1(0)}(u-v_1)(x)<0 $$
and the set
$A_0=\{x\in\R^N\setminus B_1(0): \ u(x)-v_1(x)<\frac{l_0}2\}$
is a nonempty, $C^2$- bounded open set in $\R^N\setminus \overline{B_1(0)}$.
We see that for any $x\in A_0$, $0<u(x)<v_1(x)\le (c_2 q)^{\frac1{1-q}}  |x|^{\tau_0}$,
so by the fact that $t\mapsto f(r,t)$ is decreasing in
$(0,\, (c_2 q)^{\frac1{1-q}}  r^{\tau_0})$, we obtain that
$$c_2|x|^{\alpha-N}u^q(x)-u(x)\ge c_2|x|^{\alpha-N}v_1^q(x)-v_1(x)$$
and
$$-\Delta (u-v_1-\frac{l_0}{2})\ge 0\quad {\rm in} \ \ A_0,\quad u-v_1-\frac{l_0}{2}=0\quad {\rm on}\ \ \partial A_0.$$
By Maximum Principle, we have that $u\ge v_1+\frac{l_0}{2}$ in $A_0$, which is impossible with the definition of $A_0$.
This ends the proof of the lemma.
\end{proof}

\begin{remark}\label{re 2.1}
 Notice that the decay estimate   (\ref{e 1})  holds without the restrictions (\ref{q}), so   (\ref{e 1})
could provide a lower decay estimate at infinity  in the case of (\ref{1.2}).
\end{remark}

The following proposition is an improvement of the decay of $u$ at infinity.

\begin{proposition}\label{pr e 2}
 Assume that    $p>0,\,  q\in(0, 1)$ verify (\ref{q}).
 Let   $\{\tau_j\}_j$ defined in (\ref{2.1}) with $\tau_0$ given by (\ref{t0}),  $u$ be a classical
solution of (\ref{eq 1.1}) satisfying
$$u(x)\ge b_j|x|^{\tau_{j}},\quad \forall x\in \R^N\setminus \overline{B_1(0)}$$
for some $b_j>0$ and $j\le j_0-1$.
Then   there exists $b_{j+1}>0$   such that
$$u(x)\ge b_{j+1} |x|^{\tau_{j+1}},\quad \forall x\in \R^N\setminus \overline{B_1(0)}.$$

\end{proposition}
\begin{proof}
{\it Step 1. } We prove that there exists $a_j>0$ such that
$$I_\alpha[u^p](x)\ge a_j \min\{1,\, |x|^{\alpha+\tau_jp}\},\qquad\forall x\in \R^N\setminus\{0\}.$$
 We observe that for $|x|\ge 1$
\begin{eqnarray*}
  I_\alpha[u^p](x) &\ge&  b_j^p\int_{\R^N\setminus B_1(0)}  |x-y|^{\alpha -N}|y|^{\tau_jp}dy \\
    &=& b_j^p |x|^{\alpha+\tau_jp}  \int_{ \R^N\setminus B_{\frac1{|x|}}(0) } |e_x-z|^{\alpha -N}|z|^{\tau_jp}dz
    \\
    &\ge& 2^{\alpha-N} b_j^p |x|^{\alpha+\tau_jp} \int_{ \R^N\setminus B_2(0) }|z|^{\alpha -N}|z|^{\tau_jp}dz
\\&  \ge&  a_j |x|^{\alpha+\tau_jp},
\end{eqnarray*}
where $e_x=\frac{x}{|x|}$.
Together with (\ref{2.4-3}) for $|x|<1$, we conclude that
$$I_\alpha[u^p](x)\ge a_j \min\{1,\, |x|^{\alpha+\tau_jp}\},\qquad\forall x\in \R^N\setminus\{0\}.$$

 Let
 $$v_t(x)=\arraycolsep=1pt\left\{
\begin{array}{lll}
 \displaystyle   t  |x|^{\tau_{j+1}} \quad
 &{\rm if}\quad  |x|\ge 1,\\[2mm]
 \phantom{  }
  0  \quad
 &{\rm if}\quad  |x|< 1,
\end{array}
\right.
$$
where $t>0$ will be chosen later.

{\it Step 2. } We claim that there exists $t_1>0$ such that for $t\in(0,t_1)$£¬
 \begin{equation}\label{eq 2.01j}
 \displaystyle    -\Delta v_t \le a_j|x|^{\alpha +\tau_jp}v_t^q-v_t\quad
  {\rm in}\quad  \mathbb{R}^N\setminus \overline{B_1(0)}.
 \end{equation}
For $|x|>1$,
 \begin{eqnarray*}
  a_j|x|^{\alpha +\tau_jp} v_t^q-v_t  & =&  t^q\left[a_{j}|x|^{\alpha+p\tau_j+q\tau_{j+1}}-t^{1-q}|x|^ {\tau_{j+1}} \right]   \\
  &\ge& \frac12 a_jt^q |x|^{\alpha+p\tau_j+q\tau_{j+1}},
\end{eqnarray*}
where   the last inequality  requires that
$$t\le  \left(\frac{a_j}2\right)^{\frac1{1-q}}\quad {\rm and }\quad \alpha+p\tau_j+q\tau_{j+1}=\tau_{j+1}.$$
On the other hand, we have that
$$-\Delta v_t(x)\le t |\tau_j(\tau_j+N-2)| |x|^{\tau_{j+1}-2} \quad{\rm for}\quad |x|> 1.$$
Then we obtain (\ref{eq 2.01j}) if we choose
$$  t_1= \min\left\{\left(\frac{a_j}{2|\tau_j(\tau_j+N-2)|+1}\right)^{\frac1{1-q}},\,\left(\frac{a_j}2\right)^{\frac1{1-q}},\, (a_{j} q)^{\frac1{1-q}},\, \min_{x\in \partial B_1(0)} u(x)\right\}.$$

{\it Step 3. } We prove finally that
\begin{equation}\label{2.4-4}
 u\ge v_t \quad {\rm in}\ \ \R^N\setminus \overline{B_1(0)}.
\end{equation}
The proof is very similar to prove (\ref{2.4-2}).
Let $f(r,t)=a_jr^{\alpha+\tau_jp} t^q-t$, then
$t\mapsto f(r,t)$ is decreasing in
$(0,\, (a_j q)^{\frac1{1-q}}  r^{\tau_{j+1}})$.
If (\ref{2.4-4}) fails, since $u\ge v_t $ on $\partial B_1(0)$,  we may assume that there exists some point in $\R^N\setminus \overline{B_1(0)}$ such that
$u(x)-v_t (x)<0$, then
$$l_j=\inf_{x\in\R^N\setminus B_1(0)}(u-v_t )(x)<0 $$
and the set
$A_j=\{x\in\R^N\setminus \overline{B_1(0)}: \ u(x)-v_t (x)<\frac{l_j}2\}$
is a nonempty, $C^2$- bounded open set in $\R^N\setminus \overline{B_1(0)}$.
We see that for any $x\in A_j$, $0<u(x)<v_t (x)\le (a_j q)^{\frac1{1-q}}  |x|^{\tau_{j+1}}$,
so
$$a_j|x|^{\alpha+\tau_jp}u^q(x)-u(x)\ge a_j|x|^{\alpha+\tau_jp}v_t^q(x)-v_t(x)$$
and
$$-\Delta (u-v_t-\frac{l_j}2)\ge 0\quad {\rm in} \ \ A_j,\quad u-v_t-\frac{l_j}2=0\quad {\rm on}\ \ \partial A_j.$$
By Maximum Principle, we have that $u\ge v_t+\frac{l_j}2$ in $A_j$, which is impossible with the definition of $A_j$.
\end{proof}

\noindent{\bf Proof of Theorem \ref{teo 0}.}
By contradiction, we assume that (\ref{eq 1.1}) has a nonnegative and nontrivial classical solution   $u\ge 0$, then $I_\alpha[u^p]$ is well defined in $\R^N\setminus\{0\}$.
From Proposition \ref{pr e 1}, we have that
$$u(x)\ge b_0|x|^{\tau_0}\quad{\rm for}  \ \ |x|>1.$$
Under the assumptions of Theorem \ref{teo 0}, we may repeat  Proposition  \ref{pr e 2}, we
can improve the decay estimate such that
$$u(x)\ge b_{j_0}|x|^{\tau_{j_0}}\quad{\rm for}  \ \ |x|>1,$$
where $\tau_{j_0}\ge -\frac{\alpha}{p}$.
Now fix a point $x_0\in \R^N$ with $|x_0|=1$ and we have that
\begin{eqnarray*}
I_\alpha[u^p](x_0)  \ge  b_{j_0}^p2^{\alpha-N}\int_{\R^N\setminus B_2(0)} |y|^{\alpha-N+\tau_{j_0} p}dy=\infty
\end{eqnarray*}
by  $\alpha+\tau_{j_0} p\ge0$.
That is impossible, since $u$ is a classical solution of (\ref{eq 1.1}) in $\R^N\setminus\{0\}$.\hfill$\Box$


\hskip0.5cm
\section{Classification and Existence}
\subsection{Classification}

The classification of singularities of (\ref{eq 1.1})  follows  \cite[Theorem 1.1]{CZ} where we have used essentially the following result.

\begin{theorem}\label{teo 2.1}\cite[Theorrem 1.1]{BL}
Assume that $f\in L^1_{loc}(\R^N)\cap C^1(\R^N\setminus \{0\})$
and  $u\in L^1_{loc}(\R^N)$ is  a positive classical solution of
\begin{equation}\label{eq 2.1}
 -\Delta   u+  u =f\quad
  {\rm in}\quad  \mathbb{R}^N\setminus\{0\}.
\end{equation}

Then  there exists $k\ge0$
  such that  $u$ is a distributional solution of
 \begin{equation}\label{eq 2.2}
   -\Delta u+u=f+k\delta_0\quad
  {\rm in}\quad \R^N,
\end{equation}
that is,
$$
     \int_{\R^N} [u (-\Delta  \xi)
     +u\xi-f\xi]\, dx=k\xi(0),\quad \forall \xi\in C^\infty_c(\R^N).
$$

\end{theorem}

\begin{lemma}\label{lm 2.2}
 Assume that $p>0,$ $q>0$ and $u$ is a positive classical solution of (\ref{eq 1.1}) satisfying $u^p\in L^1_{loc}(\R^N)$ and $u^p(x)\le  c_0  |x| ^{-\bar \alpha}$ for   $|x|>1$,
 where $\bar \alpha\in(\alpha,N)$.
  Then $u\in L^1_{loc}(\R^N)$ and
\begin{equation}\label{lp}
 I_\alpha[u^p]u^q\in L^1_{loc}(\R^N)\cap L^\infty (\R^N\setminus B_1(0)).
\end{equation}

 \end{lemma}
\begin{proof} From \cite[Lemma 2.1]{CZ}, we know that $ I_\alpha[u^p]u^q\in L^1_{loc}(\R^N)$.  For $|x|>2$ large enough, we have that
\begin{eqnarray*}
I_\alpha[u^p](x) &=& \int_{\R^N}\frac{u^p(y)}{|x-y|^{N-\alpha}}dy   \\
    &\le& \int_{\R^N\setminus B_1(0)}\frac{|y|^{-\bar\alpha}}{|x-y|^{N-\alpha}}dy+\int_{  B_1(0)}\frac{u^p(y)}{|x-y|^{N-\alpha}}dy
    \\&\le & |x|^{\alpha-\bar \alpha}\int_{\R^N\setminus B_{\frac1{|x|}}(0)} \frac{|z|^{-\bar\alpha}}{|e_x-z|^{N- \alpha}}dz + (|x|-1)^{\alpha-N}\int_{ B_1(0)}u^p(y)dy
    \\&\le &  |x|^{\alpha-\bar \alpha}\int_{\R^N } \frac{|z|^{-\bar\alpha}}{|e_x-z|^{N- \alpha}}dz +|x|^\alpha \norm{u}^p_{L^\infty(\R^N\setminus B_{|x|-r}(0))},
\end{eqnarray*}
where
\begin{eqnarray*}
\int_{\R^N } \frac{|z|^{-\bar\alpha}}{|e_x-z|^{N- \alpha}}dz  &\le &2^{N-\alpha} \int_{B_{1/2}(0) }  |z|^{-\bar\alpha}dz + 2^{\bar\alpha}\int_{B_{1/2}(e_1) } \frac1{|e_1-z|^{N- \alpha}}dz\\&&+\int_{\R^N}\frac1{1+|z|^{N-\alpha+\bar\alpha}}dz
\\&<&+\infty,
\end{eqnarray*}
then we have that
$$\lim_{|x|\to+\infty} I_\alpha[u^p](x)=0, $$
then it deduces that $I_\alpha[u^p]u^q\in L^\infty (\R^N\setminus B_1(0))$.
\end{proof}

\smallskip\noindent{\bf Proof of Theorem \ref{teo 1}. }  Let $u$ be a nonnegative classical solution of (\ref{eq 1.1}).  If  $u^p\in L^1(\R^N)$,   repeat the proof of \cite[Theorem 1.1]{CZ}, then we can obtain Theorem \ref{teo 1} directly. If $u^p\in L^1_{loc}(\R^N)$, we apply
 Lemma \ref{lm 2.2} to derive that $I_\alpha[u^p]u^q\in L^1_{loc}(\R^N)\cap L^\infty(\R^N\setminus B_1(0)),$ the by applying Theorem \ref{teo 2.1} with $f=I_\alpha[u^p]u^q$, we have that
there exists $k\ge 0$ such that
$$ \int_{\R^N}\left [u (-\Delta \xi)+u\xi-I_\alpha[u^p]u^q\xi\right]\, dx=k\xi(0),\quad \forall \xi\in C^\infty_c(\R^N).$$
Let $\mathbb{G}$   the Green's operator defined by   the Green kernel of $  -\Delta+Id $ in $\R^N\times\R^N$, then by Lemma \ref{lm 2.2},
we have that $\mathbb{G}[I_\alpha[u^p]u^q]$ is well-defined.  The remaining proof  is similar to  \cite[Theorem 1.1]{CZ} and we omit here.\hfill$\Box$

\subsection{Existence}

The key point for the proof of the
existence is to construct a suitable upper bound. To this end, we introduce some auxiliary functions.
Let $$\phi(r)=\left(r^{2-N}+r^{\tau}\right)e^{-\frac{r^2}{2N}}\ \ {\rm and}\ \ \varphi(r)=(r_0+|x|)^{\tau_0}, $$
where $r_0=2\sqrt{\tau_0(\tau_0-1)}$ and $\tau\in \left(2-N, \min\{0, \; 2+\alpha-(N-2)(p+q)\}\right),$
where we have $2+\alpha-(N-2)(p+q)>2-N$ by assumption that $p+q <\frac{N+\alpha}{N-2}$.

By direct computation, we have that
\begin{eqnarray*}
  \phi''(r)+\frac{N-1}{r}\phi'(r)&=&\left[\tau(\tau+N-2) r^{\tau-2} -\frac{N+2\tau}{N}r^\tau+\frac{1}{N^2}r^{\tau+2} \right]e^{-\frac{r^2}{2N}}
  \\&&+\left[ \frac{N-4}{N}r^{2-N} +\frac{1}{N^2}r^{4-N}\right]e^{-\frac{r^2}{2N}}
\end{eqnarray*}
 and
\begin{eqnarray*}
  \varphi''(r)+\frac{N-1}{r} \varphi'(r)
   &=&\tau_0(\tau_0-1)(r_0+|x|)^{\tau_0-2}+\tau_0\frac{N-1}{|x|}(r_0+|x|)^{\tau_0-1} \\
   & \le &  \tau_0(\tau_0-1) (r_0+|x|)^{\tau_0-2}
   \\&\le&  \tau_0(\tau_0-1) \frac1{r_0^2}(r_0+|x|)^{\tau_0}
   \\&=&\frac14 \varphi(r),
\end{eqnarray*}
then there exists $a_0>0$ such that
\begin{equation}\label{3.1}
 [\phi''(r)+\frac{N-1}{r}\phi'(r)]+a_0[\varphi''(r)+\frac{N-1}{r} \varphi'(r)]\le \frac12 [\phi(r)+a_0\varphi(r)],\quad \forall r>0.
\end{equation}

For $k >0$, we define
\begin{equation}\label{wk}
 w_{k}(x)= k[\phi(|x|)+a_0\varphi(|x|)],\quad \forall x\in\R^N\setminus\{0\}.
\end{equation}

Now we prove the following
\begin{lemma}\label{lm 4.3}
Assume that $(p,\,q)$ satisfies (\ref{1.2}) and $w_k$ is defined in (\ref{wk}).
Then there exists $k_0$ such that for all $0<k\le k_0$,
$w_{k}$ is a super solution of
\begin{equation}\label{eq 4.1}
 -\Delta u+u=I_\alpha[u^p]u^q\quad{\rm in}\quad \R^N\setminus \{0\}.
\end{equation}
\end{lemma}
 \begin{proof}
 From (\ref{3.1}),  we have that
\begin{equation}\label{4.4}
 -\Delta w_k(x)+w_k(x)\ge\frac12w_k(x)+c(\tau)k r^{\tau-2}e^{-\frac{r^2}{2N}},
\end{equation}
where $c(\tau)=-\tau(N-2+\tau)>0$.

 We observe that
\begin{eqnarray*}
 I_\alpha[ \varphi^p](x)= \int_{\R^N}\frac{1}{|x-y|^{N-\alpha}}(r_0+|y|)^{\tau_0p}dy:= \int_{\R^N}I(x,y)dy
\end{eqnarray*}
and   for $|x|> 2r_0$, we have that
\begin{eqnarray*}
\int_{B_{1}(0)}I(x,y)dy&\le& 2^{N-\alpha}|x|  ^{\alpha-N}\int_{B_{1}(0)}(r_0+|y|)^{\tau_0 p}dy\le c_5|x|^{\alpha-N},
\end{eqnarray*}
\begin{eqnarray*}
\int_{B_{\frac{|x|}2}(0)\setminus B_1(0)}I(x,y)dy&\le&2^{N-\alpha}|x|  ^{\alpha-N} \int_{B_{\frac{|x|}2}(0)\setminus B_1(0)}|y|^{\tau_0p}dy\le c_5|x|^{\alpha+\tau_0p},
\end{eqnarray*}
\begin{eqnarray*}
\int_{B_{\frac{|x|}2}(x)}I(x,y)dy&\le& 2^{-\tau_0p} |x|^{\tau_0p}\int_{B_{\frac{|x|}2}(x)}\frac1{|x-y|^{N-\alpha}}dy\le c_5|x|^{\alpha+\tau_0p}
\end{eqnarray*}
and
\begin{eqnarray*}
\int_{\R^N\setminus (B_{\frac{|x|}2}(x)\cup B_{\frac{|x|}2}(0))}I(x,y)dy&\le& c_6\int_{\R^N\setminus B_{\frac{|x|}2}(0)}|y|^{\alpha-N+\tau_0p}dy\le c_6|x|^{\alpha+\tau_0p}.
\end{eqnarray*}
Thus, there exists $c_7>0$ such that
$$
I_\alpha[ \varphi^p](x)\le c_7(r_0+|x|)^{\max\left\{\alpha-N,\,\alpha+\tau_0p\right\}},\quad |x|>0.
$$

Similarly, we have that for $|x|>1$
$$
I_\alpha[ \phi^p](x)\le c_7 |x|^{ \alpha-N }.
$$
We observe that there exists $c_8>0$ such that for $0<|x|<1$,
\begin{eqnarray*}
I_\alpha[ \phi^p](x)&\le&  c_8\int_{\R^N\setminus B_2(0)}|y|^{\alpha-N+\tau}e^{-\frac{|y|^2}{2N}}dy+c_8\int_{ B_2(0)}\frac{|y|^{\tau}}{|x-y|^{N-\alpha}}dy
\\&\le& c_9+c_9|x|^{\alpha+\tau} \int_{ B_{\frac2{|x|}}(0)}\frac{|z|^{\tau}}{|e_x-z|^{N-\alpha}}dz,
\end{eqnarray*}
where $e_x=\frac{x}{|x|}$ and
\begin{eqnarray*}
\int_{ B_{\frac2{|x|}}(0)}\frac{|z|^{\tau}}{|e_x-z|^{N-\alpha}}dz &\le & c_{10}\int_{B_{\frac12}(e_x)}\frac1{|e_x-z|^{N-\alpha}}dz +c_{10}\int_{B_{\frac12}(0)} |z|^\tau dz \\&&+c_{10}\int_{B_{\frac2{|x|}}(0)}(1+|z|)^{\tau+\alpha-N}dz
 \\  &\le &  c_{11}+c_{11}|x|^{-\tau-\alpha}.
\end{eqnarray*}
Therefore,  there exists $c_{12}>0$ such that
\begin{eqnarray*}
I_\alpha[\phi^p](x) \le  c_{12}+c_{12}|x|^{\alpha+\tau p}\quad{\rm for}\ \ 0<|x|<1.
\end{eqnarray*}
By the fact that
$$(a+b)^t\le 2^{t+1} (a^t+b^t)\quad {\rm for}\quad t,\,a,\, {\rm and } \,b>0,$$
 we have that for $|x|>1$,
$$
I_\alpha[w_k^p]w_k^q(x)\le c_8k^{p+q} (r_0+|x|)^{\max\left\{\alpha-N,\,\alpha+\tau_0p\right\} +\tau_0q},
$$
and for $0<|x|<1$,
$$
I_\alpha[w_k^p]w_k^q(x)\le c_9k^{p+q}[|x|^{\alpha-(N-2)(p+q)} +1].
$$
Tighter with (\ref{4.4}), to obtain the inequality
\begin{equation}\label{d 2}
 I_\alpha[w_k^p]w_k^q(x)\le \frac12 w_k++c(\tau)k r^{\tau-2}e^{-\frac{r^2}{2N}},
\end{equation}
it requires that
$$\max\left\{\alpha-N,\,\alpha+\tau_0p\right\} +\tau_0q \le \tau_0\quad{\rm and}\quad \alpha-(N-2)(p+q)\ge \tau-2,$$
which is equivalent to
\begin{equation}\label{d 1}
  \tau_0\le -\frac{\alpha}{p+q-1} \quad{\rm and}\quad \tau\le 2+\alpha-(N-2)(p+q).
\end{equation}
The first inequality of (\ref{d 1}) reads as
$$ \max\{N-2,\, \frac{N-\alpha}{1-q}\}\ge \frac{\alpha}{p+q-1},$$
which  is   (\ref{1.2}).
The first inequality of (\ref{d 1}) holds by the choosing of $\tau$.
Now (\ref{d 2}) holds if
$$c_9k^{p+q}\le (\frac12+c(\tau))k,$$
 thus, there exists $k_0>0$ such for $k\in(0,k_0)$ the following inequality holds
 \begin{eqnarray*}
 -\Delta w_k(x)+w_k(x) &\ge& \frac12 w_k+c(\tau) |x|^{\tau-2}e^{-\frac{|x|^2}{2N}} \\
     &\ge & I_\alpha[w_k^p]w_k^q(x),\quad \forall\, |x|>0.
 \end{eqnarray*}
This completes the proof.
\end{proof}



\medskip

\noindent {\bf Proof of Theorem \ref{teo 2} for existence part.}  For any $k \in (0, k_0)$ where $k_0$ is given by Lemma \ref{lm 4.3}, we define the
iterating sequence $\{v_n\}_{n \geq 0}$
by $$v_0:= k \Gamma_0>0$$
and
$$
 v_n  =  \mathbb{G}[I_\alpha[v_{n-1}^p] v_{n-1}^q]+ k \Gamma_0.
$$
Observing that
$$v_1= \mathbb{G} [I_\alpha[v_0^p] v_0^q ] + k \Gamma_0>v_0$$
and  assuming that
$$
v_{n-1} \ge  v_{n-2} \quad{\rm in} \quad \R^N\setminus\{0\},
$$
we deduce that
\begin{eqnarray*}
 v_n =   \mathbb{G}[I_\alpha[v_{n-1}^p] v_{n-1}^q]+ k \Gamma_0
 \ge  \mathbb{G}[I_\alpha[v_{n-2}^p] v_{n-2}^q]+ k \Gamma_0
 =   v_{n-1}.
\end{eqnarray*}
Thus, the sequence $\{v_n\}_n$ is a increasing with respect to $n$.
Moveover, we have that
\begin{equation}\label{4.2.3}
\int_{\R^N} v_n[(-\Delta \xi) +\xi] \,dx =\int_{\R^N} I_\alpha[v_{n-1}^p] v_{n-1}^q\xi \,dx +k\xi(0), \quad \forall \xi\in C^\infty_c(\R^N).
\end{equation}

From Lemma \ref{lm 4.3},
 and the definition of $w_k$, we have  $w_k>v_0$ and
$$v_1= \mathbb{G} [I_\alpha[v_0^p]v_0^q]+ k \Gamma_0\le  \mathbb{G} [I_\alpha[w_k^p]w_k^q]+ k w_k=w_k.$$
Inductively, we obtain
\begin{equation}\label{2.10a}
v_n\le w_k
\end{equation}
for all $n\in\N$. Therefore, the sequence $\{v_n\}_n$ converges. Let $u_{k}:=\lim_{n\to\infty} v_n$. By \eqref{4.2.3}, $u_{k}$ is a weak solution of (\ref{eq 1.2})
and satisfies (\ref{1.2}).

For $k\le k_q$, we have that $u_k\le w_{t_q}$ in $\R^N\setminus\{0\}$, so  $u_k\in L^p(\R^N)$ and $I_\alpha[u_k^p]u_k^q$ is bounded uniformly locally in $\R^N\setminus\{0\}$,
then $u_k$ is a classical solution of (\ref{eq 1.1}).

We claim that $u_{k}$ is the minimal solution of (\ref{eq 1.1}), that is, for any positive solution $u$ of (\ref{eq 1.2}), we always have $u_{k}\leq u$. Indeed,  there holds
\[
 u  = \mathbb{G}[ I_\alpha[ u^p]u^q]+ k \Gamma_0\ge v_0,
\]
and then
\[
 u  = \mathbb{G}[ I_\alpha[ u^p]u^q]+ k \Gamma_0\ge  \mathbb{G}[ I_\alpha[ v_0^p]v_0^q]+ k \Gamma_0=v_{1}.
\]
We may show inductively that
\[
u\ge v_n
\]
for all $n\in\N$.  The claim follows.
 \hfill$\Box$

 \subsection{Decay at infinity}

It remains to prove
the decay at infinity of $u_k$ satisfies (\ref{1.8}) or (\ref{1.10}) under the conditions (\ref{1.6}) and (\ref{1.9}) respectively.
From  Lemma \ref{lm 4.3} and Proposition \ref{pr e 1}, we have that   the minimal solution $u_k$ of (\ref{eq 1.2}) has the decay
 \begin{equation}\label{d 3}
  \limsup_{|x|\to+\infty}u_k(x)|x|^{-\tau_0}\le k\quad{\rm and}\quad  \liminf_{|x|\to+\infty}u_k(x)|x|^{-\tau_0}>0.
 \end{equation}
So  when $(p,\,q)$ satisfies
 \begin{equation}\label{1.2<}
 (1-\frac\alpha N)p+q>1 \quad{\rm and }\quad p+q <  \frac{N+\alpha}{N-2},
 \end{equation}
then the minimal solution   $u_k^p(x)\le c|x|^{\tau_0p}$ with $-\tau_0p>N$ and $u^p\in L^1(\R^N)$.  In this case, we employee  the idea in \cite{MZ} to
refine the decay estimate in the case (\ref{1.6}).

When
 \begin{equation}\label{1.2>}
 (1-\frac\alpha N)p+q\le  1 \quad{\rm and }\quad 1+\frac{\alpha}{N-2} \le p+q <  \frac{N+\alpha}{N-2},
 \end{equation}
we see that  $u_k^p$ is no longer in $L^1(\R^N)$ and
$$u_k^p(x)\le 3k |x|^{-(N-2)p},\qquad |x|>1, $$
where we note that $(N-2)p\in (\alpha,N)$.

We first prove the following

\begin{proposition}\label{pr 4.1}
Assume that $\alpha\in(0,\, N)$, $\beta>N$ and the nonnegative function $f\in L^1(\R^N)$ satisfies
\begin{equation}\label{6.1}
f_\infty:=\sup_{|x|>r_0} f(x)|x|^\beta<+\infty.
\end{equation}
 Let $I_\alpha(x) = \frac{1}{|x|^{N-\alpha}}$, then for $|x|>r_0$ large,
 \begin{equation}\label{6.2}
\left|I_\alpha [ f](x)-\norm{f}_{L^1(\R^N)}I_\alpha(x)\right|\le  \frac{c_{10}}{|x|^{N-\alpha+\gamma}},
 \end{equation}
 holds for $\gamma=\frac{\beta-N}{1+\beta-N}$ and
 $c_{10} >0$ depending on $\alpha,\,\beta,\, N$  and $f_\infty$.

\end{proposition}
\begin{proof}
For $|x|>r_0$ large enough and $R=|x|^{1- \gamma}$, we have that
\begin{equation}\label{6.3}
\frac{1}{(|x|+R)^{N-\alpha}} \int_{B_R(0)}f(y) dy \le\int_{B_R(0)}\frac{f(y)}{|x-y|^{N-\alpha}}dy\le \frac{1}{(|x|-R)^{N-\alpha}}\int_{B_R(0)}f(y) dy.
\end{equation}
We observe that
$$\frac{1}{(|x|-R)^{N-\alpha}}\le \frac{1}{|x| ^{N-\alpha}}\left(1+c_{11}\frac{R}{|x|}\right),$$
$$\frac{1}{(|x|+R)^{N-\alpha}}\ge \frac{1}{|x| ^{N-\alpha}}\left(1-c_{11}\frac{R}{|x|}\right)$$
and
$$
\int_{\R^N\setminus B_R(0)}f(y) dy  \le   f_\infty  \int_{\R^N\setminus B_R(0)} |y|^{-\beta} dy =  c_{12}f_\infty R^{N-\beta},
$$
where $c_{11},\, c_{12}>0$ are independent of $R$.
That means
$$\left| \int_{ B_R(0)}f(y) dy-\norm{f}_{L^1(\R^N)}\right| \le c_{12}f_\infty R^{N-\beta}. $$
Therefore, from (\ref{6.3}) we have that
\begin{equation}\label{6.4}
\left| \int_{B_R(0)}\frac{f(y)}{|x-y|^{N-\alpha}}dy- \norm{f}_{L^1(\R^N)}I_\alpha(x) \right|\le  \frac{c_{13}}{|x|^{N-\alpha+\gamma}}.
\end{equation}
On the other hand,
\begin{eqnarray*}
 \int_{\R^N\setminus B_R(0)}\frac{f(y)}{|x-y|^{N-\alpha}} dy   &\leq & f_\infty  \int_{\R^N\setminus B_R(0)}\frac{1}{|x-y|^{N-\alpha}|y|^{\beta}} dy    \\
    &=&  f_\infty \frac{1}{|x|^{\beta-\alpha}} \int_{\R^N\setminus B_{\frac{R}{|x|}}(0)}\frac{1}{|e_x-z|^{N-\alpha}|z|^{
    \beta}} dz,
\end{eqnarray*}
where $e_x=\frac{x}{|x|}$ and
\begin{eqnarray*}
&&\int_{\R^N\setminus B_{\frac{R}{|x|}}(0)}\frac{1}{|e_x-z|^{N-\alpha}|z|^{\beta}} dz
\\& &\le
  2^{ \beta}\int_{B_\frac12(e_x) }\frac{1}{ |e_x-z|^{N-\alpha}} dz+  2^{N-\alpha}\int_{\R^N\setminus B_{\frac{R}{|x|}}(0) }\frac{1}{ |z|^{\beta}} dz
   \\& &\le c_{14} \left(\frac{R}{|x|}\right)^{N-\beta},
\end{eqnarray*}
where we have used the fact that the first term $\int_{B_\frac12(e_x) }\frac{1}{ |e_x-z|^{N-\alpha}} dz$ is bounded. Thus,
 \begin{equation}\label{6.5}
 \int_{\R^N\setminus B_R(0)}\frac{f(y)}{|x-y|^{N-\alpha}} dy    \le c_{14} f_\infty \frac{1}{|x|^{\beta-\alpha}} \left(\frac{R}{|x|}\right)^{N-\beta} = \frac{c_{14}f_\infty}{|x|^{N-\alpha+\gamma}},
 \end{equation}
 since $\beta -\alpha + \gamma (N- \beta) = N -\alpha +\gamma$. Therefore
the estimates (\ref{6.4}) and (\ref{6.5}) imply (\ref{6.2}).
\end{proof}

Before to complete the proof of Theorem \ref{teo 2}, we need also the following
\begin{lemma}\label{lm 4.4}\cite[Lemma 6.7]{MV}
Assume that $u$  is a solution of the equation
$$
 \arraycolsep=1pt
\begin{array}{lll}
 \displaystyle \ \   -\Delta   u+  \mu u =\frac{\nu}{|x|^\sigma}\quad
  {\rm in}\quad  \mathbb{R}^N\setminus B_r(0),\\[2mm]
 \phantom{ }
   \displaystyle    \lim_{|x|\to+\infty}u(x)=0,
\end{array}
 $$
where parameters  $\mu,\,\nu, \,\sigma,\, r$ are positive.
Then we have $$\lim_{|x|\to+\infty}u(x)|x|^{\sigma}=\frac{\nu}{\mu}. $$
\end{lemma}

\noindent {\bf Proof of (\ref{1.8}) and (\ref{1.10}) in Theorem \ref{teo 2}. }
  (\ref{1.10}) follows by (\ref{d 3}). Next we prove (\ref{1.8}).

{\it Step 1.} From Lemma \ref{lm 4.3}, we have that
$$u_k(x)\le k|x|^{-\frac{N-\alpha}{1-q}}\quad{\rm for}\quad |x|>r_0.$$
From $(1-\frac\alpha N)p+q>1$, we have that
$\frac{N-\alpha}{1-q}p>N$ and
$$u_k^p(x)\le k^p |x|^{-\frac{N-\alpha}{1-q}p} \le k^p |x|^{-N} \quad{\rm for}\quad |x|>r_0.$$
Moreover, $u_k^p\in L^1(\R^N)$ by Lemma \ref{lm 4.3}.
From Proposition \ref{pr 4.1} with $\gamma=\frac{\frac{N-\alpha}{1-q}p-N}{1+\frac{N-\alpha}{1-q}p-N}$, there exists $c_{15}>0$ such that
\begin{equation}\label{5.2}
\norm{u_k^p}_{L^1(\R^N)}I_\alpha(x)- \frac{c_{15} }{|x|^{N-\alpha+\gamma}}\le  I_\alpha[u_k^p](x)\le \norm{u_k^p}_{L^1(\R^N)}I_\alpha(x)+ \frac{ c_{15}}{|x|^{N-\alpha+\gamma}},\quad |x|>r.
\end{equation}
Let $w=u_k^{1-q}$, then
$$-\Delta w+(1-q)w\ge (1-q)\norm{u_k^p}_{L^1(\R^N)} I_\alpha(x)- c_{16} \frac1{|x|^{N-\alpha+\gamma}}, $$
Then by Lemma \ref{lm 4.4} and Comparison Principle for $-\Delta +(1-q)$, we obtain that
$$ \liminf_{|x|\to+\infty} w(x)|x|^{N-\alpha}\ge \norm{u_k^p}_{L^1(\R^N)}. $$
Then
$$\liminf_{|x|\to+\infty} u_k(x)|x|^{\frac{N-\alpha}{1-q}}\ge \norm{u_k^p}^{\frac1{1-q}}_{L^1(\R^N)}.$$

{\it Step 2.} From the Young's inequality, we have that
$$I[u_k^p]u_k^q\le (1-q)I[u_k^p]^{\frac{1}{1-q}}+ qu_k$$
and then
\begin{eqnarray*}
-\Delta u_k+(1-q)u_k &\le & (1-q)I[u_k^p]^{\frac{1}{1-q}} \\
   &\le & (1-q)\norm{u_k^p}^{\frac{1}{1-q}}_{L^1(\R^N)}|x|^{-\frac{N-\alpha}{1-q}}(x)+ c_{17}|x|^{-\frac{N-\alpha+\gamma}{1-q}},
\end{eqnarray*}
then by Lemma \ref{lm 4.4} and comparison principle for $-\Delta +(1-q)$, we have that
$$ \limsup_{|x|\to+\infty} u_k(x)|x|^{N-\alpha}\le \norm{u_k^p}^{\frac{1}{1-q}}_{L^1(\R^N)}.$$

Therefore, we conclude that
$$ \lim_{|x|\to+\infty} u_k(x)|x|^{N-\alpha}= \norm{u_k^p}^{\frac{1}{1-q}}_{L^1(\R^N)}$$
and this completes the proof of Theorem \ref{teo 2}.
\hfill$\Box$\smallskip
\medskip

\bigskip

\noindent\thanks{Acknowledgements: H. Chen  is   partially supported by NNSF of China, No:11401270,  by SRF for ROCS, SEM and by the Jiangxi Provincial Natural Science Foundation, No: 20161ACB20007.
 F. Zhou is partially supported by NSFC (11271133 and 11431005)
and Shanghai Key Laboratory of PMMP.}

\end{document}